\documentclass[12pt]{article}
\usepackage{amsmath,amssymb,amsthm}
\usepackage{graphicx}
\usepackage{booktabs}
\usepackage{cite}
\usepackage{url}
\usepackage{mathrsfs}
\usepackage{bm}
\usepackage{xcolor}

\usepackage[top=2.5cm, bottom=2.5cm, left=2.5cm, right=2.5cm]{geometry}

\newtheorem{theorem}{Theorem}[section]
\newtheorem{lemma}[theorem]{Lemma}

\newtheorem{corollary}[theorem]{Corollary}

\usepackage{hyperref}

\title{A new proof of Poincar\'e-Miranda theorem based on the classification of one-dimensional manifolds}
\author{Xiao-Song Yang\\
	1) School of Mathematics and Statistics\\
	Huazhong University of Science and Technology\\
	2) Hubei Key Laboratory of Engineering Modeling and Scientific Computing\\
	Huazhong University of Science and Technology, Wuhan 430074, China}
\date{}

\begin{document}
	
	\maketitle
	
	\begin{abstract}
		This note gives a new elementary proof of Poincar\'e-Miranda theorem based on Sard's theorem and the simple classification of one-dimensional manifolds.
	\end{abstract}
	
	
	
	\section{Introduction}
	
	The classical Poincar\'e-Miranda theorem states that given any continuous map from a cuboid in $\mathbb{R}^{n}$ into $\mathbb{R}^{n}$, if its $j$th component takes constant opposite signs on the corresponding opposite $j$th-faces of the cuboid ($j=1,\cdots,n$), then it has at least one zero in this cuboid.
	
	This theorem was stated in 1883 by J. H. Poincar\'e \cite{poincare1883}, with a hint of proof. Later C. Miranda \cite{miranda1940} proved the equivalence of this theorem with Brouwer's fixed point theorem. Thus a rigorous proof of the Poincar\'e-Miranda theorem was given based on the well known Brouwer's Fixed Point Theorem. In 2013, an elementary proof based on the basic properties of differential forms in $\mathbb{R}^{n}$ and Stokes formula on a cuboid was obtained by J. Mawhin \cite{mawhin2013}.
	
	In the present note, we will give a new elementary proof of the above theorem just by means of Sard's theorem \cite{milnor1997} and the simple classification of one-dimensional manifolds.
	
	\section{On zeros of smooth maps from a cuboid to $\mathbb{R}^{n}$}
	
	In this section we consider the zeros of a smooth map from an $n$-dimensional cuboid to $\mathbb{R}^{n}$. We need two elementary obvious facts.
	
	\begin{lemma}\label{lemma1}
		Consider a $C^{1}$ function $f:[a,b]\to\mathbb{R}^{1}$, suppose $|f^{\prime}(x)|\neq 0$ for every $x\in\mathbb{R}^{1}$ with $f(x)=0$. If $f(a)\cdot f(b)<0$, then the function $f$ has an odd number of zeros within the interval $[a,b]$; If $f(a)\cdot f(b)>0$, then the function $f$ has an even number of zeros within the interval $[a,b]$, including the possibility of no zeros.
	\end{lemma}
	
	\begin{lemma}\label{lemma2}
		Every $C^{1}$ function $f:S^{1}\to\mathbb{R}^{1}$ has even number of zeros if $|f^{\prime}(x)|\neq 0$ for each $x$ with $f(x)=0$.
	\end{lemma}
	
	Now let us recall the statement of the Poincar\'e-Miranda theorem. On zeros of smooth maps from $n$-dimensional cube to $\mathbb{R}^{n}$, let $I^{n}=\prod_{i=1}^{n}[a_{i},b_{i}]$ be an $n$-dimensional cube, let $f:I^{n}\to\mathbb{R}^{n}$ be continuous map satisfying the following boundary condition \eqref{condition1}:
	
	For $f(x)=(f_{1}(x),f_{2}(x),\cdots,f_{n}(x))$,
	\begin{equation}\label{condition1}
		f_{i}(x_{1},\cdots,a_{i},\cdots,x_{n})\cdot f_{i}(x_{1},\cdots,b_{i},\cdots,x_{n})<0,\quad i=1,2,\cdots,n
	\end{equation}
	
	Then the well known Poincar\'e-Miranda theorem states that there is at least one zero of $f$ within the $n$-dimensional cube $I^{n}$. If $f$ is a smooth map, then there is more to say about the zeros of $f$. In this note we will show that in a generic sense, a smooth map satisfying condition \eqref{condition1} has an odd number of zeros in $I^{n}$, the proof will be carried out by elementary techniques from differential topology. First we have the following statement.
	
	\begin{theorem}\label{theorem1}
		Suppose a smooth map $f:I^{n}\rightarrow\mathbb{R}^{n}$ satisfies condition \eqref{condition1}. Then for any $\epsilon>0$, there is a point $q\in\mathbb{R}^{n}$ with $||q||<\epsilon$ such that the set $f^{-1}(q)=\{x\in I^{n},f(x)=q\}$ has an odd number of points, i.e. $\#(f^{-1}(q))\equiv 1\pmod{2}$.
	\end{theorem}
	
	\begin{proof}
		We prove the theorem by induction on $n$.
		
		\textbf{STEP 1:} For $n=1$, one has $f:[a_{1},b_{1}]\rightarrow\mathbb{R}^{1}$. In view of Sard's Theorem, for any $\epsilon>0$, there is $q\in\mathbb{R}^{1}$ satisfying $|q|<\epsilon$, $|q|<\min\{|f(a_{1})|,|f(b_{1})|\}$ such that $f^{-1}(q)$ is nonempty and $q$ is a regular value of $f$. Since $|f^{\prime}(x)|\neq 0$ for each $x\in f^{-1}(q)$. One can easily prove the statement as in Lemma \ref{lemma1}.
		
		\textbf{STEP 2:} For $n=2$, let $f(x)=(f_{1}(x),f_{2}(x))$, $x=(x_{1},x_{2})\in I^{2}$. By Sard's Theorem, for any $\epsilon>0$, there is a point $q=(q_{1},q_{2})\in\mathbb{R}^{2}$ which is a regular value of $f$ and satisfies
		\[
		||q||<\epsilon,
		\]
		and
		\[
		|q_{1}|<\min\{\hat{f}_{1}(a_{1}),\hat{f}_{1}(b_{1})\},\quad |q_{2}|<\min\{\hat{f}_{2}(a_{2}),\hat{f}_{2}(b_{2})\},
		\]
		where
		\begin{align*}
			\hat{f}_{1}(a_{1})&=\min_{x_{2}\in[a_{2},b_{2}]}|f_{1}(a_{1},x_{2})|,\quad \hat{f}_{1}(b_{1})=\min_{x_{2}\in[a_{2},b_{2}]}|f_{1}(b_{1},x_{2})|, \\
			\hat{f}_{2}(a_{2})&=\min_{x_{1}\in[a_{1},b_{1}]}|f_{2}(x_{1},a_{2})|,\quad \hat{f}_{2}(b_{2})=\min_{x_{1}\in[a_{1},b_{1}]}|f_{2}(x_{1},b_{2})|.
		\end{align*}
		Clearly $q_{1}$ is also a regular value of $f_{1}(x)$. This combined with boundary condition \eqref{condition1} implies nonemptiness of $f_{1}^{-1}(q_{1})$ and $f_{2}^{-1}(q_{2})$, and $f_{1}^{-1}(q_{1})$ is a 1-dimensional closed submanifold of $I^{2}$, therefore $f_{1}^{-1}(q_{1})$ consists of circles or closed segments with their boundaries contained in $\partial I^{2}$, precisely in $[a_{1},b_{1}]\times \{a_{2}\} \cup [a_{1},b_{1}]\times \{b_{2}\}$. First we prove that $f_{1}^{-1}(q_{1})$ contains at least one segment connecting $[a_{1},b_{1}]\times \{a_{2}\}$ and $[a_{1},b_{1}]\times \{b_{2}\}$ as shown in Figure \ref{DV_1}. Since $q_{1}$ is regular for $f_{1}$ both on $I^{2}$ and in restriction to $\partial I^{2}$, $f_{1}$ has an odd number of points in $[a_{1},b_{1}]\times \{a_{2}\}$ and $[a_{1},b_{1}]\times \{b_{2}\}$ such that $f_{1}=q_{1}$. This implies that not every line segment has both of its endpoints contained in $[a_{1},b_{1}]\times \{a_{2}\}$ or contained in $[a_{1},b_{1}]\times \{b_{2}\}$, because the subset of points in $[a_{1},b_{1}]\times\{a_{2},b_{2}\}$ that satisfy $f_{1}=q_{1}$ also belong to $f_{1}^{-1}(q_{1})$. Therefore $f_{1}^{-1}(q_{1})$ must have at least one segment connecting $[a_{1},b_{1}]\times \{a_{2}\}$ and $[a_{1},b_{1}]\times \{b_{2}\}$. In fact, from the above argument it is easy to see that $f_{1}^{-1}(q_{1})$ contains an odd number of segments that connect $[a_{1},b_{1}]\times \{a_{2}\}$ and $[a_{1},b_{1}]\times \{b_{2}\}$.
		
		\begin{figure}
			\centering
			\includegraphics[width=0.7\linewidth]{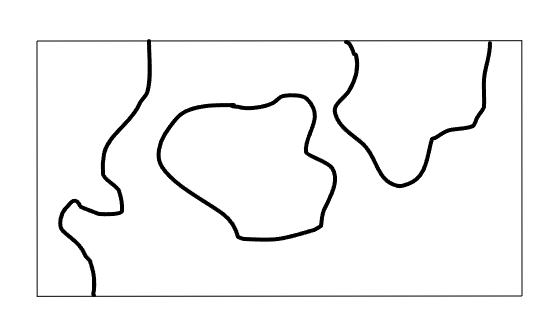}
			\caption{The image of $f_{1}^{-1}(q_{1})$.} 
			\label{DV_1}
		\end{figure}
		
		Now consider $f_{2}(x)$. Let $h$ be a connected component of $f_{1}^{-1}(q_{1})$. If $h$ is a segment connecting $[a_{1},b_{1}]\times \{a_{2}\}$ and $[a_{1},b_{1}]\times \{b_{2}\}$, then $f_{2}(x)|_{h}\rightarrow\mathbb{R}^{1}$ has an odd number of points $\bar{x}$ in $h$ such that $f_{2}(\bar{x})=q_{2}$. If $h$ is a circle then $f_{2}$ has an even number of points $\bar{x}$ in $h$ with $f_{2}(\bar{x})=q_{2}$ because of regularity of $q$ for $f=(f_{1},f_{2})$. The same conclusion also holds for $h$ only intersecting $[a_{1},b_{1}]\times \{a_{2}\}$ or $[a_{1},b_{1}]\times \{b_{2}\}$. In other words, there are an even number of points such that $f_{2}(\bar{x})=q_{2}$ in this case.
		
		Putting these together, we see that $f_{2}$ has an odd number of points at which $f_{2}$ equals $q_{2}$. Thus $f$ has an odd number of points at which $f$ equals $q=(q_{1},q_{2})$, and we have proved the theorem in case $n=2$.
		
		\textbf{STEP 3:} Now assuming the theorem holds for $n=m$, let's consider the map $f:I^{m+1}\rightarrow\mathbb{R}^{m+1}$.
		
		For $f=(f_{1},f_{2},\cdots,f_{m+1})$, let $\bar{f}:I^{m+1}\rightarrow\mathbb{R}^{m}$, where
		\[
		\bar{f}(x) = (f_{1}(x),f_{2}(x),\cdots,f_{m}(x)).
		\]
		
		For any $\epsilon>0$, let $q=(q_{1},q_{2},\cdots,q_{m+1})$ be a regular value of $f$, such that
		\[
		||q||<\epsilon,
		\]
		and
		\[
		|q_{i}|<\min\{\hat{f}_{i}(a_{i}),\hat{f}_{i}(b_{i})\},\quad i=1,2,\cdots,m+1,
		\]
		where
		\begin{align*}
			\hat{f}_{i}(a_{i})&=\min\{|f_{i}(x_{1},\cdots,x_{i},\cdots,x_{m+1})| : x\in I^{m+1},\ x_{i}=a_{i}\}, \\
			\hat{f}_{i}(b_{i})&=\min\{|f_{i}(x_{1},\cdots,x_{i},\cdots,x_{m+1})| : x\in I^{m+1},\ x_{i}=b_{i}\}.
		\end{align*}
		
		Since $\bar{q}=(q_{1},q_{2},\cdots,q_{m})$ is also a regular value for $\bar{f}$, so $\bar{f}^{-1}(\bar{q})$ is a 1-dimensional closed submanifold of $I^{m+1}$ and its boundary points are contained in $\partial I^{m+1}$ precisely in
		\[
		\{x\in I^{m+1}\mid x_{m+1}=a_{m+1}\} \cup \{x\in I^{m+1}\mid x_{m+1}=b_{m+1}\} = (B^{-}_{m+1} \cup B^{+}_{m+1}),
		\]
		where
		\begin{align*}
			B^{-}_{m+1}&=\{x\in I^{m+1}\mid x_{m+1}=a_{m+1}\}, \\
			B^{+}_{m+1}&=\{x\in I^{m+1}\mid x_{m+1}=b_{m+1}\}.
		\end{align*}
		
		Clearly every component $h$ of $\bar{f}^{-1}(\bar{q})$ is a segment or a circle. We know from the inductive hypothesis that $\bar{f}$ restricted to $B^{-}_{m+1}$ or $B^{+}_{m+1}$ has an odd number of points at which $\bar{f}$ equals $\bar{q}$, thus, there is an odd number of segments that connect $B^{-}_{m+1}$ and $B^{+}_{m+1}$.
		
		Let $h$ be a connected component of $\bar{f}^{-1}(\bar{q})$. If $h$ is a segment connecting $B^{-}_{m+1}$ and $B^{+}_{m+1}$, then from Lemma \ref{lemma1}, the map $f_{m+1}|_{h}\rightarrow\mathbb{R}$ has an odd number of points at which $f_{m+1}$ equals $q_{m+1}$; and $f_{m+1}$ has an even number of points at which $f_{m+1}$ equals $q_{m+1}$ in the other two cases as discussed in the $n=2$ case. It is clear that the sum of all the points of $\bar{f}^{-1}(\bar{q})$ at which $f_{m+1}$ equals $q_{m+1}$ is an odd number and all of these points are just the set $f^{-1}(q)$.
		
		The proof is completed.
	\end{proof}
	
	By means of the same argument in Theorem \ref{theorem1}, it is easy to see the following statement holds.
	
	\begin{corollary}\label{corollary1}
		Suppose that smooth maps from the cube to $\mathbb{R}^{n}$ satisfy the Poincar\'e-Miranda boundary conditions. If the zero is a regular value of $f$, then the number of zeros of $f$ is odd.
	\end{corollary}
	
	From Corollary \ref{corollary1}, we can have a theorem that is of a flavor of Poincar\'e-Hopf theorem on vector field.
	
	\begin{theorem}\label{theorem2}
		Let $v(x)=(v_{1}(x),v_{2}(x),\cdots,v_{n}(x))$ be a smooth vector field on the cube $I^{n}=[-1,1]\times[-1,1]\times\cdots\times[-1,1]$. Suppose $v$ points outwards on the boundary $\partial I^{n}$, and every zero of $v$ is non-degenerate, then the number of zeros of $v$ is odd.
	\end{theorem}
	
	\begin{proof}
		Since the vector field points outwards, we can derive that
		\[
		v_{i}(x_{1},\cdots,-1,\cdots,x_{n})\cdot v_{i}(x_{1},\cdots,1,\cdots,x_{n})<0,\quad i=1,2,\cdots,n,
		\]
		which satisfies the boundary condition \eqref{condition1}. According to Corollary \ref{corollary1}, the theorem holds.
	\end{proof}
	
	\section{A proof of Poincar\'e-Miranda theorem}
	
	\begin{proof}
		Let $f:I^{n}\rightarrow\mathbb{R}^{n}$ be a continuous map satisfying the boundary condition \eqref{condition1}. Suppose that $f$ has no zero point on $I^{n}$. Then compactness of $I^{n}$ implies existence of $\eta>0$ such that $||f(x)||>\eta$ for all $x\in I^{n}$. Let $g:I^{n}\rightarrow\mathbb{R}^{n}$ be a smooth map satisfying boundary condition \eqref{condition1} as well as
		\[
		||f(x)-g(x)||<\frac{\eta}{2},\quad \forall x\in I^{n},
		\]
		therefore
		\[
		||g(x)||>\frac{\eta}{2},\quad \forall x\in I^{n}.
		\]
		
		In view of Theorem \ref{theorem1}, there is a regular value $c$ for $g$ such that $||c||<\frac{\eta}{4}$, and $g^{-1}(c)$ is nonempty. Now for $x\in g^{-1}(c)$, we have $||g(x)||<\frac{\eta}{4}$, leading to a contradiction.
	\end{proof}

\end{document}